\documentclass{article}
\usepackage{amsmath, amsthm, amsfonts}

\newtheorem{theorem}{Theorem}[section]
\newtheorem{problem}{Problem}[section]

\newtheorem{lemma}[theorem]{Lemma}

\theoremstyle{definition}

\theoremstyle{remark}
\newtheorem{remark}[theorem]{Remark}

\begin{document}

\title{A Stochastic production planning problem}
\author{Elena Cristina Canepa \\
{\small University Politehnica of Bucharest, Bucharest, Romania.}\\
{\small E-mail address: cristinacanepa@yahoo.com} \and Dragos-Patru Covei \\
{\small The Bucharest University of Economic Studies, Bucharest, Romania. }
\and {\small E-mail address: dragos.covei@csie.ase.ro} \and Traian A. Pirvu 
\\
{\small McMaster University, Hamilton, Canada. }\\
{\small E-mail address: tpirvu@math.mcmaster.ca} }
\maketitle

\abstract{Stochastic production planning problems were studied in several works; the
model with one production good was discussed in \cite{BS}. The extension to
several economic goods is not a trivial issue as one can see from the recent
works \cite{CDPAMC}, \cite{CP} and \cite{GK}. The following qualitative
aspects of the problem are analyzed in \cite{CP}; the existence of a
solution and its characterization through dynamic programming/HJB equation,
as well as the verification (i.e., the solution of the HJB equation yields
the optimal production of the goods). In this paper, we stylize the model of 
\cite{CDPAMC} and \cite{CP} in order to provide some quantitative answers to
the problem. This is possible especially because we manage to solve the HJB
equation in closed form. Among other results, we find that the optimal
production rates are the same across all the goods and they also turn to be
independent of some model parameters. Moreover we show that production rates
are increasing in the aggregate number of goods produced, and they are also
uniformly bounded. Numerical experiments show some patterns of the output.}

\section{Introduction}

Production planning problems were studied for quite some time. \cite{TS}
considered a stochastic production-inventory model to determine optimal
production rates, i.e., the ones which minimize a discounted quadratic loss
function. Their solution has three terms: the initial inventory, a steady
state of the solution and a correction term which kicks in when time
approaches maturity. This work was extended from a deterministic to a
stochastic framework by \cite{BS} and \cite{ST} who added randomness to the
dynamics of the inventory process. The work of \cite{FSS} looks at the
infinite horizon stochastic production planning problem in which a
continuous-time Markov chain models the demand.

The aforementioned papers consider in general the production planning
problem with one economic good only. The extension to several economic goods
makes the problem more mathematically involved as one can see from the
recent works of \cite{CDPAMC} and \cite{CP}. Moreover, \cite{CP}
characterized the solution through dynamic programming/HJB equation; using
regularity and estimate results from the area of partial differential
equations a classical solution of the HJB was established, and the
verification result was proved. Since these works deal with the infinite
horizon, a transversality condition was imposed on the value function, and
it was shown that the value function verifies it. The paper \cite{GK} is
within the paradigm of multiple goods' production. Because of the complexity
of HJB equations, the goal is not to solve the HJB equations, but to offer
an approximate solution.

In this paper we specialized the model of \cite{CDPAMC} and \cite{CP} to
make it more tractable and to obtain quantitative results. Our main
contribution is that we solved in closed form the HJB equation and the
optimal production rate. The solution displays a mean field structure; the
optimal production rate of some good is a function of the number of that
specific produced good and an average of all the goods produced (this
average is expressed by a norm of the vector of goods produced). By
exploiting the structure of our closed form solution we can see that the
optimal production rates are the same across all goods and they do not
depend on some model parameters. Moreover, the optimal production rates are
zeros when there are no goods produced, and they are of order $O(\frac{1}{N}%
) $ ($N$ here stands for the number of goods). We show that production rates
are increasing in the aggregate number of goods produced, and they are also
uniformly bounded. Numerical experiments reveal that the production rate is
a decreasing function of the number of goods' type $N$ and, the variance of
the number of goods produced.

Finally, the HJB equation characterizing the optimal production rates
appears in other practical applications as we mention in the last section of
the paper.

Now we are ready to present the organization of this paper. Section \ref%
{sec2} describes the model. Section \ref{sec3} provides the methodology.
Section \ref{sec4} presents other practical applications of the mathematics
developed. The paper ends with an appendix containing a technical proof.

\section{The model \label{sec2}}

Consider a factory producing $N$ types of economic goods which stores them
in an inventory designated place. Next, we describe the model
mathematically. There exists a complete probability space $(\Omega ,\mathcal{%
F},\{\mathcal{F}_{t}\}_{0\leq t\leq \infty },P),$ on which lives a $N$%
-dimensional Brownian motion denoted by $w=\left( w_{1},...,w_{N}\right) $.

The filtration $\{\mathcal{F}_{t}\}_{0\leq t\leq \infty },$ is the natural
filtration of the Brownian motion. Let $p\left( t\right) =\left(
p_{1}(t),...,p_{N}(t)\right) $, represent the production rate at time $t$
(control variable). Next, let us introduce the control variables. Let the
threshold $p^{0}=\left( p_{1}^{0},...,p_{N}^{0}\right) $ be a vector
standing for the factory optimal production level. This level can be optimal
from a technological standpoint, but its implementation may not be optimal
because of inventory costs.

Next, let $l=\left( l_{1},...,l_{N}\right) $ be the factory-optimal
inventory level which can be attained but not maintained since there is
noise in the system. In order to simplify the notations we assume that $%
p^{0}=l=\left( 0,...,0\right) $. This simplification is obtained by
considering deviations from the factory-optimal inventory level and the
factory-optimal production level. The deviations may be negative.

Next, let us describe the inventories. There exists a constant demand rate
for every economic good, demand rate represented by the vector $\xi =\left(
\xi _{1},...,\xi _{N}\right) $. Again, to simplify the notations we take $%
\xi =\left( 0,...,0\right) $ meaning that we consider deviations from the
constant demand rate.

Let $y_{i}^{0}$ denote the initial inventory level of good $i$, and $%
y_{i}(t) $ the inventory level of good $i$, at time $t$. These inventory
levels are modelled by the following system of stochastic differential
equations

\begin{equation}
dy_{i}\left( t\right) =(p_{i}-\xi _{i})dt+\sigma dw_{i}\text{, }y_{i}\left(
0\right) =y_{i}^{0}\text{, }i=1,...,N,  \label{cons0}
\end{equation}%
where $\sigma $ is a constant (non-zero) diffusion coefficient. Let us
recall that the stochasticity here is due to inventory spoilages which are
random in nature.

Let $\tau $ be the stopping time representing the moment when the inventory
level reaches some threshold $R$, i.e., 
\begin{equation*}
\tau =\inf_{t>0}\{\left\vert y(t)\right\vert \geq R\}.
\end{equation*}%
Here, $|\cdot |$ stands for the Euclidian norm, and this way of limiting the
inventory level is imposed for tractability. The factory may consider
stopping the production when the inventory level $R$ is attained and/or
exceeded.

\subsection{The Objective}

The performance over time of a production $p\left( t\right) =\left(
p_{1}(t),...,p_{N}(t)\right)$ is measured by means of its cost. At this
point we introduce the cost functional which yields the cost: 
\begin{equation}
J\left( p_{1},...,p_{N}\right) :=\text{ }E\int_{0}^{\tau } (|p(t)|^2+
|y(t)|^2)dt,  \label{opti}
\end{equation}%
which measures the quadratic loss. Again let us recall that we measure
deviations from an optimal state, whence the loss. At this point we are
ready to frame our objective, which is to minimize the cost functional.
i.e., 
\begin{equation}
\inf \{J\left( p_{1},...,p_{N}\right) \left\vert \, p_{i} \text{, }\forall
i=1,2,...,N\right. \}\text{, }  \label{min}
\end{equation}
subject to the It\u{A}%
\'{}
equation (\ref{cons0}).

\section{The Methodology \label{sec3}}

Having presented the problem we want to solve, now we provide our means to
tackle it. Our approach is based on the value function and dynamic
programming which leads to the HJB equation. Let $z$ denote the value
function, i.e., 
\begin{equation*}
z(y_{1}^{0},y_{2}^{0},\ldots ,y_{N}^{0})=\inf \{J\left(
p_{1},...,p_{N}\right) \left\vert \,p_{i}\text{, }\forall i=1,2,...,N\right.
\}\text{, }
\end{equation*}%
subject to the It\u{A}%
\'{}
equation (\ref{cons0}). We apply probabilistic techniques to characterize
the value function; that is we search for a function $U\left( x\right) $
such that the stochastic process $M^{p}(t)$ defined below 
\begin{equation*}
M^{p}\left( t\right) =U\left( y\left( t\right) \right)
-\int_{0}^{t}[f_{1}(p(s))+f_{2}(y(s))]\,ds,
\end{equation*}%
is supermartingale for all $p\left( t\right) =\left(
p_{1}(t),...,p_{N}(t)\right) $ and martingale for the optimal control $%
p^{\ast }\left( t\right) =\left( p_{1}^{\ast }(t),...,p_{N}^{\ast
}(t)\right) $. Once such a function is found it turns out that $-U=z$. We
search for $U$ a $C^{2}\left[ 0,R\right] $ function and the
supermartingale/martingale requirement yields by means of Ito's Lemma the
Hamilton-Jacobi-Bellman (HJB) equation which characterizes the value
function 
\begin{equation}
-\frac{\sigma ^{2}}{2}\Delta z-\left\vert x\right\vert ^{2}=\inf \{p\nabla
z+\left\vert p\right\vert ^{2}\left\vert \text{ }\forall i=1,2,...,N\right.
\}.  \label{solv}
\end{equation}%
This HJB can be turned into a partial differential equation (PDE) since a
simple calculation yields 
\begin{equation}
\inf \{p\nabla z+\left\vert p\right\vert ^{2}\left\vert p_{i}\text{ }\forall
i=1,2,...,N\right. \}=-\frac{1}{4}\left\vert \nabla z\right\vert ^{2}.
\label{foc}
\end{equation}%
Thus, the HJB equation becomes the PDE 
\begin{equation*}
-\frac{\sigma ^{2}}{2}\Delta z-\left\vert x\right\vert ^{2}=-\frac{1}{4}%
\left\vert \nabla z\right\vert ^{2}\text{ for }x\in \mathbb{R}^{N}\text{, }%
|x|\leq R\text{,}
\end{equation*}%
or, equivalently 
\begin{equation}
2\sigma ^{2}\Delta z+4\left\vert x\right\vert ^{2}=\left\vert \nabla
z\right\vert ^{2}\text{ for }x\in \mathbb{R}^{N}\text{, }|x|\leq R\text{.}
\label{eq}
\end{equation}%
The change of variable $z=-v$, yields the PDE 
\begin{equation}
\Delta v=\frac{4\left\vert x\right\vert ^{2}-\left\vert \nabla v\right\vert
^{2}}{2\sigma ^{2}}\text{ for }x\in \mathbb{R}^{N}\text{, }|x|\leq R\text{.}
\label{las2}
\end{equation}%
The gradient term in the above PDE can be removed by the change of variable $%
u\left( x\right) =e^{\frac{v\left( x\right) }{2\sigma ^{2}}}$, to get a
simpler PDE 
\begin{equation}
\left\{ 
\begin{array}{l}
\Delta u\left( x\right) =\frac{1}{\sigma ^{4}}\left\vert x\right\vert
^{2}u\left( x\right) \text{ for }x\in \mathbb{R}^{N}\text{, }|x|\leq R\text{%
, } \\ 
u\left( x\right) >0\text{ for }x\in \mathbb{R}^{N}\text{, }|x|\leq R.%
\end{array}%
\right.   \label{I}
\end{equation}

The value function will give us in turn the candidate optimal control. The
first order optimality conditions on the lefthand side of (\ref{foc}) are
sufficient for optimality since we deal with a quadratic (convex) function
and they produce the candidate optimal control as follows:

\begin{equation*}
p_{i}^{\ast }=\overline{p}_{i}(y_{1}\left( t\right) ,\ldots ,y_{N}\left(
t\right) )\text{, }i=1,...,N,
\end{equation*}%
and 
\begin{equation}
\overline{p}_{i}=\frac{1}{2}\frac{\partial v}{\partial x_{i}}\left(
x_{1},...,x_{N}\right) \text{, for }i=1,...,n\text{.}  \label{optio}
\end{equation}

\subsection{The Equation of Value Function \label{ms}}

Let $B_{R}\left( 0\right) $ be the ball in $\mathbb{R}^{N}$ centered at the
origin and radius $R>0$. The equation of the value function according to (%
\ref{I}) is 
\begin{equation}
\Delta u\left( x\right) =\frac{1}{\sigma ^{4}}\left\vert x\right\vert
^{2}u\left( x\right) \text{ in }B_{R}\left( 0\right) .\text{{}}
\label{master}
\end{equation}%
The boundary condition is taken to be 
\begin{equation}
u(0)=\alpha ,  \label{radial}
\end{equation}%
where $\alpha $ is a positive constant. The following result concerns the
equation of value function.

\begin{theorem}
\label{ode_sol} Given the positive constant $\alpha $, there exists a unique
positive radially symmetric solution $u_{\alpha }\in C^{2}\left[ 0,R\right] $%
, to the problem (\ref{master}) subject to the Dirichlet boundary condition
( \ref{radial}). Moreover, the solution is convex and increasing, and the
following holds true 
\begin{eqnarray}
u_{\alpha }^{\prime }(0) &=&0\text{,}  \label{prop1} \\
u_{\alpha }\left( r\right) &=&\alpha \left( 1+\underset{j=1}{\overset{\infty 
}{\Sigma }}\frac{1}{j!\left( N+2\right) \left( N+6\right) ...(N+4j-2)}\left( 
\frac{r^{2}}{2\sigma ^{2}}\right) ^{2j}\right) \text{, }  \label{prop4} \\
u_{\alpha }^{\prime }\left( r\right) &=&\alpha \underset{j=1}{\overset{%
\infty }{\Sigma }}\frac{4jr}{2\sigma ^{2}j!\left( N+2\right) \left(
N+6\right) ...(N+4j-2)}\left( \frac{r^{2}}{2\sigma ^{2}}\right) ^{2j-1},
\label{derivat}
\end{eqnarray}%
for all $r:=\left\vert x\right\vert \in \lbrack 0,R]$. In addition, 
\begin{eqnarray}
u_{\alpha }\left( r\right) &\leq &\alpha e^{\frac{r^{4}}{4\sigma ^{4}(N+2)}}%
\text{, }r\in \lbrack 0,R]\text{,}  \label{prop2} \\
\left( u_{\alpha }\right) ^{\prime }(r) &\leq &\frac{\alpha r^{3}}{\sigma
^{4}(N+2)}e^{\frac{r^{4}}{4\sigma ^{4}(N+2)}}\text{, }r\in \lbrack 0,R]\text{%
,}  \label{prop3}
\end{eqnarray}%
hold. \ 
\end{theorem}

\subparagraph{\textbf{Proof.} It is done in the appendix}

\subsection{Verification}

In this subsection we show that the control of (\ref{cons}) is indeed
optimal. In a first step let us show that $M^{p}(t)$ 
\begin{equation*}
M^{p}\left( t\right) =U\left( y\left( t\right) \right)
-\int_{0}^{t}(|p(s)|^{2}+|y(s)|^{2})\,ds,
\end{equation*}%
is supermartingale for all 
\begin{equation*}
p\left( t\right) =\left( p_{1}(t),...,p_{N}(t)\right)
\end{equation*}%
and martingale for the optimal control 
\begin{equation*}
p^{\ast }\left( t\right) =\left( {p}_{1}^{\ast }(t),...,{p}_{N}^{\ast
}(t)\right) .
\end{equation*}%
Indeed, Ito%
\'{}
Lemma yields for the optimal control candidate 
\begin{equation*}
dM^{p}\left( t\right) =(\frac{\sigma ^{2}}{2}\Delta U(y(s))-\left\vert
y(s)\right\vert ^{2}+p(s)\nabla U(s)-\left\vert p(s)\right\vert
^{2})ds+\sigma p(s)\nabla {z}(y(s))dw(s).
\end{equation*}%
Then, the claim yields in light of HJB equation (\ref{solv}).

In a second step let us establish the optimality of $\left( p_{1}^{\ast
},...,p_{N}^{\ast }\right) $. The martingale/supermartingale principle
yields 
\begin{equation*}
EU\left( y^{\ast }\left( \tau \ast \right) \right) -E\int_{0}^{{\tau }\ast
}(|p^{\ast }(u)|^{2}+|y^{\ast }(u)|^{2})du=U(y^{\ast }\left( {0}\right)
)=U(y\left( {0}\right) ),
\end{equation*}%
and%
\begin{equation*}
EU\left( y^{{}}\left( \tau \right) \right) -E\int_{0}^{\tau
}(|p^{{}}(u)|^{2}+|y^{{}}(u)|^{2})du\leq U(y\left( {0}\right) ).
\end{equation*}%
Here, let us recall that $\tau \ast =\inf_{t>0}\{\left\vert y^{\ast
}(t))\right\vert \geq R\}$ and $\tau =\inf_{t>0}\{\left\vert {y}%
(t)\right\vert \geq R\}$. Moreover, 
\begin{equation*}
EU\left( y^{\ast }\left( {\tau }\ast \right) \right) =EU\left( y^{{}}\left(
\tau \right) \right) =2\sigma ^{2}{\ln {u(R)}},
\end{equation*}%
and this finishes the proof.

\subsection{Optimal Control}

Let us notice that equations (\ref{optio}) become 
\begin{equation}
\overline{p}_{i}(y_{1},\ldots ,y_{N})=\sigma ^{2}\frac{u_{\alpha }^{\prime
}(r)}{ru_{\alpha }(r)}y_{i}\text{, }r\neq 0\text{, }i=1,2\cdots N,
\label{optcont1}
\end{equation}%
and $r=|y|$. The optimal control is given by 
\begin{equation*}
{p}_{i}^{\ast }=\overline{p}_{i}(y_{1}\left( t\right) ,\ldots ,y_{N}\left(
t\right) )\text{, }i=1,...,N,
\end{equation*}%
and 
\begin{equation}
dy_{i}\left( t\right) =p_{i}^{\ast }dt+\sigma dw_{i}\text{, }y_{i}\left(
0\right) =y_{i}^{0}\text{, }i=1,...,N.  \label{cons}
\end{equation}%
This SDE system has a unique solution since the map $y\rightarrow \bar{p}%
_{i}(y)$, $i=1,...,N$, is Lipschitz on $[0,R]$. Let us notice that the
production rate 
\begin{equation}
\frac{\overline{p}_{i}}{y_{i}}=\sigma ^{2}\frac{u_{\alpha }^{\prime }(r)}{%
ru_{\alpha }(r)},\text{ }r\neq 0,  \label{optcont}
\end{equation}%
is the same across \textbf{all} goods. Let us notice the connection with 
\textbf{mean field models}, with the key quantity being $r=|y|$.

\begin{remark}
\label{unic}The choice of $\alpha >0$ is irrelevant because the value
function equation admits the following symmetry; if $u$ is the solution with 
$\alpha =1$, then $\alpha u$ is the solution for arbitrary $\alpha >0$.
However, both $u$ and $\alpha u$ yield the same optimal control (see (\ref%
{optcont})). Let us notice that if we impose the boundary condition $\bar{u}%
\left( R\right) =\alpha >0$ instead of (\ref{radial}) then we get a solution 
$\bar{u}$ which is a scalar multiple of $u$, i.e., $\bar{u}=Ku$, for some
constant $K>0$. Thus, $\bar{u}$ yields the same optimal control (see (\ref%
{optcont})). Therefore, the optimal control does not depend on the choices
of $\alpha $ and $R.$
\end{remark}

In light of this remark we set $\alpha =1$, so that 
\begin{equation}
u\left( r\right) :=u_{1}\left( r\right) =1+\underset{j=1}{\overset{\infty }{%
\Sigma }}\frac{1}{j!\left( N+2\right) \left( N+6\right) ...(N+4j-2)}\left( 
\frac{r^{2}}{2\sigma ^{2}}\right) ^{2j}\text{, }  \label{closedform}
\end{equation}%
for all $r\geq 0$, whence we can get the production rate $\sigma ^{2}\frac{%
u^{\prime }(r)}{ru(r)}$, $r\neq 0$, in closed form. Moreover, from (\ref%
{prop4}) we get that $\lim_{r\rightarrow 0}\frac{u^{\prime }(r)}{ru(r)}=0$,
thus the optimal production rates are zeros when there are no goods produced.

Using (\ref{closedform}) and operations with power series (see \cite{DS}
Chapter 1), we get the optimal production rate in closed form.

\begin{theorem}
The optimal production rate is given by 
\begin{equation*}
\frac{\bar{p}_{i}}{y_{i}}=\sigma ^{2}\frac{u^{\prime }(r)}{ru(r)}=\frac{%
4\sigma ^{2}}{r^{2}}\overset{\infty }{\underset{j=0}{\Sigma }}c_{j}\left[ 
\frac{r^{4}}{4\sigma ^{4}}\right] ^{j},\text{ }r\neq 0,
\end{equation*}%
where 
\begin{eqnarray*}
a_{0} &=&1\text{, }c_{0} =0\text{, }c_{j}=\frac{1}{a_{0}}\left[ b_{j}-%
\overset{j}{\underset{i=1}{\Sigma }}c_{j-i}a_{i}\right] \text{, }j=1,2,3,...
\\
a_{j} &=&\frac{1}{j!\left( N+2\right) \left( N+6\right) ...(N+4j-2)}\text{, }%
j=1,2,... \\
b_{j} &=&\frac{j}{j!\left( N+2\right) \left( N+6\right) ...(N+4j-2)}\text{, }%
j=1,2,...
\end{eqnarray*}
\end{theorem}

The production rate is increasing and bounded. This fact will me made
precise in the following Lemma.

\begin{lemma}
\label{mono} The function 
\begin{equation*}
r\rightarrow \frac{u^{\prime }(r)}{ru(r)}
\end{equation*}
is increasing and 
\begin{equation}
\frac{u^{\prime }(r)}{ru(r)}\leq \frac{1}{\sigma ^{2}}.  \label{ineq}
\end{equation}
\end{lemma}

\begin{proof}
The first part of the claim yields if the derivative of this function is
positive which boils down to 
\begin{equation*}
u^{\prime \prime }\left( r\right) \geq \frac{(u^{\prime }(r))^{2}}{u(r)}+%
\frac{u^{\prime }(r)}{r}.
\end{equation*}%
Next we use the fact that $u$ solves the following ODE 
\begin{equation*}
u^{\prime \prime }\left( r\right) +\frac{N-1}{r}u^{\prime }\left( r\right) =%
\frac{1}{\sigma ^{4}}r^{2}u\left( r\right) ,
\end{equation*}%
whence, the claim becomes 
\begin{equation*}
\frac{1}{\sigma ^{4}}r^{2}u\left( r\right) \geq N\frac{u^{\prime }(r)}{r}+%
\frac{(u^{\prime }(r))^{2}}{u(r)}.
\end{equation*}%
This is equivalent to 
\begin{equation*}
\frac{u^{\prime }(r)}{u(r)}\leq \frac{\sqrt{\frac{N^{2}}{r^{2}}+\frac{4r^{2}%
}{\sigma ^{4}}}-\frac{N}{r}}{2},
\end{equation*}%
or 
\begin{equation}
\frac{u^{\prime }(r)}{ru(r)}\leq \frac{\sqrt{\frac{N^{2}}{r^{2}}+\frac{4r^{2}%
}{\sigma ^{4}}}-\frac{N}{r}}{2r}.  \label{estim}
\end{equation}%
This argument shows that 
\begin{equation*}
r\rightarrow \frac{u^{\prime }(r)}{ru(r)},
\end{equation*}%
is increasing if and only if (\ref{estim}) holds true. However, the function 
\begin{equation*}
r\rightarrow \frac{\sqrt{\frac{N^{2}}{r^{2}}+\frac{4r^{2}}{\sigma ^{4}}}-%
\frac{N}{r}}{2r},
\end{equation*}%
is increasing, both functions are $0$ when $r=0$ (since $u^{\prime }(0)=0$)
and $r\rightarrow \frac{u^{\prime }(r)}{ru(r)}$ is increasing on some small
interval $[0,\epsilon ]$ in light of $u$ being convex (for this see Theorem %
\ref{ode_sol}). This shows that $r\rightarrow \frac{u^{\prime }(r)}{ru(r)}$
is increasing and (\ref{estim}) holds true. Moreover, since 
\begin{equation*}
r\rightarrow \frac{\sqrt{\frac{N^{2}}{r^{2}}+\frac{4r^{2}}{\sigma ^{4}}}-%
\frac{N}{r}}{2r}
\end{equation*}%
is increasing and has as asymptote at infinity $\frac{1}{\sigma ^{2}}$ we
also get the second part of the claim.
\end{proof}

\subsection{Asymptotic Analysis}

\begin{equation*}
\end{equation*}%
Let us recall the estimate for large $N$ from \cite{CDPAMC} 
\begin{equation*}
\frac{u^{\prime }(r)}{r}\leq \frac{K}{N-1}\text{, }r\neq 0.\text{{}}
\end{equation*}%
Thus, for big $N$ an approximate solution is 
\begin{equation*}
\frac{u^{\prime }(r)}{ru(r)}\leq \frac{K}{u_{0}(N-1)}\approx 0\text{, }r\neq
0\text{,}
\end{equation*}%
which says that the optimal control $p^{\ast }\approx 0$, since 
\begin{equation*}
\overline{p}_{i}=\sigma ^{2}\frac{u^{\prime }(r)}{ru(r)}x_{i}\text{, }r\neq 0%
\text{, }i=1,...,N.
\end{equation*}%
This means that if the number of goods is big then $p^{\ast }=0$ is an
approximate solution.

Next, we prove an asymptotical result.

\begin{lemma}
\label{liml} The following result hold true 
\begin{equation}  \label{lim}
\lim_{r\rightarrow\infty} \left[ \frac{u^{\prime }(r)}{ru(r)} \right] =\frac{%
1}{\sigma ^{2}}.
\end{equation}
\end{lemma}

\begin{proof}
Because the function $u$ is convex and increasing (for this see Theorem \ref%
{ode_sol}) it follows that 
\begin{equation*}
\lim_{r\rightarrow \infty }u(r)=\lim_{r\rightarrow \infty }u^{\prime
}(r)=\infty .
\end{equation*}%
In light of Lemma \ref{mono} the limit exists and is finite. Let us denote
it by $l.$ L'Hospital rule yields 
\begin{equation}
l=\lim_{r\rightarrow \infty }\left[ \frac{u^{\prime }(r)}{ru(r)}\right]
=\lim_{r\rightarrow \infty }\left[ \frac{u^{\prime \prime }(r)}{%
u(r)+ru^{\prime }(r)}\right] .  \label{L'Hos}
\end{equation}%
Next we use the fact that $u$ solves the following ODE 
\begin{equation*}
u^{\prime \prime }\left( r\right) +\frac{N-1}{r}u^{\prime }\left( r\right) =%
\frac{1}{\sigma ^{4}}r^{2}u\left( r\right) ,
\end{equation*}%
whence 
\begin{equation*}
u^{\prime \prime }\left( r\right) =\frac{1}{\sigma ^{4}}r^{2}u\left(
r\right) -\frac{N-1}{r}u^{\prime }\left( r\right) .
\end{equation*}%
Inserting this into (\ref{L'Hos}) we get 
\begin{equation*}
l=\frac{1}{l\sigma ^{4}}.
\end{equation*}%
Therefore 
\begin{equation*}
l=\lim_{r\rightarrow \infty }\left[ \frac{u^{\prime }(r)}{ru(r)}\right] =%
\frac{1}{\sigma ^{2}}.
\end{equation*}
\end{proof}

\subsection{\textbf{Simulation of the optimal inventory}}

Let us recall the SDE system 
\begin{equation}
dy_{i}\left( t\right) =p_{i}^{\ast }dt+\sigma dw_{i}\text{, }y_{i}\left(
0\right) =y_{i}^{0}\text{, }i=1,...,N,  \label{consf}
\end{equation}%
governing the optimal inventory.

This SDE system can be simulated numerically. It can be done using a Euler
scheme as follows: start with $y_{i}^{0}$, $i=1,...,N$, and 
\begin{equation*}
r=\Sigma _{i=1}^{N}\left[ y_{i}^{0}\right] ^{2}\text{, }r\neq 0\text{.}
\end{equation*}%
On $\left[ 0,\Delta t\right] $ we approximate 
\begin{equation*}
y_{i}\left( \Delta t\right) \simeq \sigma ^{2}\frac{u^{\prime }\left(
r\right) }{ru\left( r\right) }y_{i}^{0}+\sigma \sqrt{\Delta t}Z_{i}^{0}\text{%
, }r\neq 0\text{,}
\end{equation*}%
where $Z_{i}^{0}$ is standard normal.

Next repeat this on $\left[ \Delta t,2\Delta t\right] $ as follows: 
\begin{equation*}
r\left( \Delta t\right) =\Sigma _{i=1}^{N}\left[ y_{i}^{\Delta t}\right]
^{2},
\end{equation*}%
and 
\begin{equation*}
y_{i}\left( 2\Delta t\right) \simeq \sigma ^{2}\frac{u^{\prime }\left(
r\left( \Delta t\right) \right) }{r\left( \Delta t\right) u\left( r\left(
\Delta t\right) \right) }y_{i}^{\Delta t}+\sigma \sqrt{\Delta t}Z_{i}^{1}%
\text{,}
\end{equation*}%
where $Z_{i}^{1}$ is standard normal. The process is then repeated on $\left[
2\Delta t,3\Delta t\right] ,$ and so on. In the following we present two
plots resulting from this simulation procedure. We considered $N=2$ (two
economic goods) and $\sigma =2$ in the first plot $\sigma =5$ in the second
plot.

\subsection{Numerical Experiments}

In the first set of experiments we set $\sigma =0.5,$ and vary $N$ the
number of goods' type.

We observe from these set of plots the following patterns:

\begin{itemize}
\item[1)] the production rate is an increasing function of the total number
of goods produced, fact explained by Lemma \ref{mono};

\item[2)] when the total number of goods produced exceed a certain threshold
the production rate converges to $1,$ fact explained by Lemma \ref{liml};

\item[3)] the production rate is a decreasing function of the total number
of goods produced.
\end{itemize}

In the next set of plots we set $N=100,$ and vary $\sigma $.

We observe from these set of plots the following patterns:

\begin{itemize}
\item[1)] the production rate is an increasing function of the total number
of goods produced, fact explained by Lemma \ref{mono};

\item[2)] when the total number of goods produced exceed a certain threshold
the production rate converges, fact explained by Lemma \ref{liml};

\item[3)] the production rate is a decreasing function of $\sigma .$
\end{itemize}

\section{ Other Applications \label{sec4}}

The value function equation characterizing the optimal control, i.e., (\ref%
{master}), appears naturally in other practical applications. There is by
now a vast literature concerning on the existence of positive solutions and
their behaviour for the partial differential equation%
\begin{equation}
\Delta u\left( x\right) =f\left( x,u\left( x\right) \right) \text{ for }x\in
\Omega ,  \label{1}
\end{equation}%
where $\Omega $ is a bounded or unbounded domain of $\mathbb{R}^{N}$ ($N\geq
1$) or the all space $\mathbb{R}^{N}$ and $f$ is a function suitable chosen.

The interest in studying the above equation comes, for instance, from
various physical situations, such as quantum mechanics, quantum optics,
nuclear physics and reaction-diffusion processes (cf. \cite{AL,TL,P,CJJ}).
For instance, a basic preoccupation for the study of problem (\ref{1}) is
the time-independent Schr\"{o}dinger equation (single non-relativistic
particle) \ 
\begin{equation}
\Delta u=\frac{2m}{\text{ 
h{\hskip-.2em}\llap{\protect\rule[1.1ex]{.325em}{.1ex}}{\hskip.2em}%
}^{2}}(V\left( x\right) -E)u,\text{ 
h{\hskip-.2em}\llap{\protect\rule[1.1ex]{.325em}{.1ex}}{\hskip.2em}%
}=h/2\pi ,  \label{sch}
\end{equation}%
where $h$ is Planck's constant, 
h{\hskip-.2em}\llap{\protect\rule[1.1ex]{.325em}{.1ex}}{\hskip.2em}
is the reduced Planck constant (or the Dirac constant), $E$ and $V(x)$ are
the total (non relativistic) and potential energies of a particle of mass $m$%
, respectively.

Besides the importance in applications, the equation (\ref{1}) also raises
many difficult mathematical problems that need to be solved. In general, the
existence of the solutions and numerical approximation of the elliptic
problem (\ref{1}) is widely open. See the paper of Santos, Zhou and Santos 
\cite{CJJ}, which includes a nice survey and recent progresses for Eq. (\ref%
{1}).

Let us mention this result which is interesting in itself.

\begin{theorem}
(see \cite[Theorem 2.1, p. 199]{HY}) The problem (\ref{master}) subject to
the Dirichlet boundary condition 
\begin{equation}
u\left( x\right) \rightarrow \infty \text{ as }\left\vert x\right\vert
\rightarrow R\text{,}  \label{ldindat}
\end{equation}
has no positive solutions.
\end{theorem}

Even if the next result has no importance in economic theories, it helps us
to understand the beauty of this problem and to discover other questions
that will need to be solved by the researchers.

\begin{theorem}
\label{cjj}(see \cite{CJJ}) The problem (\ref{master}) with $B_{R}\left(
0\right) $ replaced with $\mathbb{R}^{N}$, admits a sequence of symmetric
radial solutions $u_{k}\left( \left\vert x\right\vert \right) \in
C^{2}\left( \mathbb{R}^{N}\right) $ with 
\begin{equation*}
u_{k}\left( 0\right) =\infty \text{ as }k\rightarrow \infty \text{.}
\end{equation*}%
Besides this, $u_{k}^{\prime }\geq 0$ in $\left[ 0,\infty \right) $.
\end{theorem}

In the next, we provide two exact solutions for the problem (\ref{master})
with $B_{R}\left( 0\right) $ replaced with $\mathbb{R}^{4}\setminus \{0_{%
\mathbb{R}^{4}}\}$. They are: 
\begin{eqnarray}
u_{1}\left( x\right) &=&\alpha e^{\frac{1}{2\sigma ^{2}}\left\vert
x\right\vert ^{2}}\left\vert x\right\vert ^{-2}\text{, }\alpha \in \mathbb{R}%
\text{ and }x\in \mathbb{R}^{4}\setminus \{0_{\mathbb{R}^{4}}\},
\label{sol1} \\
u_{2}\left( x\right) &=&\alpha e^{-\frac{1}{2\sigma ^{2}}\left\vert
x\right\vert ^{2}}\left\vert x\right\vert ^{-2}\text{, }\alpha \in \mathbb{R}%
\text{ and }x\in \mathbb{R}^{4}\setminus \{0_{\mathbb{R}^{4}}\}.
\label{sol2}
\end{eqnarray}%
The solutions (\ref{sol1}) and (\ref{sol2}) were determined by analyzing the
series in (\ref{prop4}) and can be used by physicists in the study of the
time-independent Schrodinger equation (\ref{sch}). Moreover, reasoning in
the same manner we think that similar solutions can be constructed for the
total (non relativistic) and potential energies of a particle of mass $m$ in
(\ref{sch}).

Next, we posit the following open problems inspired by the two solutions and 
\cite{CJJ}.

\begin{problem}
\label{proex1}Assume that $g\in C^{1}\left( \left[ 0,\infty \right) ,\left[
0,\infty \right) \right) $ is a non-decreasing function satisfying%
\begin{equation*}
\int_{\gamma }^{\infty }\frac{1}{\sqrt{\int_{0}^{t}g\left( s\right) ds}}%
dt=\infty \text{, for }t\geq \gamma >0\text{,}
\end{equation*}%
and $p$ is a non-negative continuous symmetric radially function such that%
\begin{equation*}
\int_{0}^{\infty }t^{1-N}\int_{0}^{t}s^{N-1}p\left( s\right) dsdt=\infty .
\end{equation*}%
Then, there exists at least one positive radially symmetric solution $u\in
C^{2}\left( \mathbb{R}^{N}\setminus \{0_{\mathbb{R}^{N}}\}\right) $ for the
problem%
\begin{equation}
\Delta u\left( x\right) =p\left( r\right) g\left( u\left( x\right) \right) 
\text{ in }\mathbb{R}^{N}\text{, }r=\left\vert x\right\vert \text{,{}}
\label{ecgen}
\end{equation}%
subject to the Dirichlet boundary condition%
\begin{equation}
u\left( x\right) \rightarrow \infty \text{ as }\left\vert x\right\vert
\rightarrow \infty \text{,}  \label{problem1}
\end{equation}%
such that%
\begin{equation*}
u\left( x\right) \rightarrow \infty \text{ as }\left\vert x\right\vert
\rightarrow 0\text{.}
\end{equation*}%
Moreover, $\partial u/\partial r\geq 0$ on $\left[ t_{0},\infty \right) $
and $\partial u/\partial r<0$ on $\left[ 0,t_{0}\right) $, for some $%
t_{0}\geq 0$.
\end{problem}

\begin{problem}
\label{proex2}Under the same assumptions on $p$ and $g$ as in Problem \ref%
{proex1}, there exists at least one positive radially symmetric solution $%
u\in C^{2}\left( \mathbb{R}^{N}\setminus \{0_{\mathbb{R}^{N}}\}\right) $ of (%
\ref{ecgen}) subject to the Dirichlet boundary condition%
\begin{equation}
u\left( x\right) \rightarrow 0\text{ as }\left\vert x\right\vert \rightarrow
\infty \text{,}  \label{problem2}
\end{equation}%
such that%
\begin{equation*}
u\left( x\right) \rightarrow \infty \text{ as }\left\vert x\right\vert
\rightarrow 0\text{.}
\end{equation*}%
Moreover, $\partial u/\partial r\leq 0$ on $\left[ 0,\infty \right) $.
\end{problem}

Example of solutions for problems \ref{proex1}, and \ref{proex2} are the
ones given in (\ref{sol1}), and (\ref{sol2}). To the best of our knowledge
the only result for the problems \ref{proex1}, and \ref{proex2} is Theorem %
\ref{cjj}.

\section{Appendix}

\subsection{Proof of Theorem \protect\ref{ode_sol}}

We consider the radial form of the problem (\ref{master}) subject to the
Dirichlet boundary condition (\ref{radial}), i.e., 
\begin{equation}
\left\{ 
\begin{array}{c}
u_{\alpha }^{\prime \prime }\left( r\right) +\frac{N-1}{r}u_{\alpha
}^{\prime }\left( r\right) =\frac{1}{\sigma ^{4}}r^{2}u_{\alpha }\left(
r\right) \text{ in }[0,R]\text{,} \\ 
u_{\alpha }\left( 0\right) =\alpha .%
\end{array}%
\right.  \label{radialpr}
\end{equation}%
We show that the solution $u_{\alpha }\left( r\right) $ of (\ref{radialpr})
can be obtained succesively in the following way 
\begin{equation}
\left\{ 
\begin{array}{ll}
u_{\alpha }^{0}\left( r\right) =u_{\alpha }\left( 0\right) =\alpha , &  \\ 
u_{\alpha }^{k}\left( r\right) =\alpha
+\int_{0}^{r}t^{1-N}\int_{0}^{t}s^{N-1}\frac{1}{\sigma ^{4}}s^{2}u_{\alpha
}^{k-1}\left( s\right) dsdt & \text{for }0<r\leq R\text{ and }k\in \mathbb{N}%
^{\ast }.%
\end{array}%
\right.  \label{numrec}
\end{equation}%
It is easy to see that \{$u_{\alpha }^{k}\left( r\right) $\}$_{k\geq 0}$ is
a nondecreasing sequence of functions satisfying%
\begin{eqnarray}
u_{\alpha }^{k+1}\left( r\right) -u_{\alpha }^{k}\left( r\right) &\leq &%
\frac{\alpha }{\left( k+1\right) !}\left( \frac{r^{4}}{4\sigma ^{4}(N+2)}%
\right) ^{k+1}  \label{1in} \\
&\leq &\frac{\alpha }{\left( k+1\right) !}\left( \frac{R^{4}}{4\sigma
^{4}(N+2)}\right) ^{k+1}\overset{k\rightarrow \infty }{\rightarrow }0,
\label{2in}
\end{eqnarray}%
for all $r\in \left[ 0,R\right] $. Then \{$u_{\alpha }^{k}\left( r\right) $\}%
$_{k\geq 0}$ is a Cauchy sequence of functions on $\left[ 0,R\right] $. It
is a straightforward argument to prove that 
\begin{equation*}
u_{\alpha }^{k}\in C^{2}\left[ 0,R\right] \text{, }k\in \mathbb{N}.
\end{equation*}%
Since a Cauchy sequence of functions is convergent, it has a limit function $%
u_{\alpha }\left( r\right) $ and the convergence is uniform. Moreover, since
an uniformly Cauchy sequence of continuous functions has a continuous limit,
then $u_{\alpha }\left( r\right) $ is a continuous function on $\left[ 0,R%
\right].$

By passing to the limit in (\ref{numrec}) we obtain that $u_{\alpha }\left(
r\right) $ verifies the integral form of the problem (\ref{master}) subject
to the Dirichlet boundary condition (\ref{radial})%
\begin{equation}
u_{\alpha }\left( r\right) =\alpha +\int_{0}^{r}t^{1-N}\int_{0}^{t}s^{N-1}%
\frac{1}{\sigma ^{4}}s^{2}u_{\alpha }\left( s\right) dsdt\text{, in }[0,R].
\label{intform}
\end{equation}%
Hence, the limit function $u_{\alpha }\left( r\right) $ is the solution of (%
\ref{master}) subject to the Dirichlet boundary condition (\ref{radial}).

Next, we examine the sequence \{$(u_{\alpha }^{k}\left( r\right) )^{\prime }$%
\}$_{k\geq 0}$. We note first that\ 
\begin{equation}
0\leq \left( u_{\alpha }^{k}\right) ^{\prime
}(r)=r^{1-N}\int_{0}^{r}t^{N-1}t^{2}u_{\alpha }^{k-1}\left( t\right) dt.
\label{deriv}
\end{equation}%
Thus, the function $r\rightarrow u_{\alpha }^{k}(r)$ is nondecreasing for
all $k\in \mathbb{N}$. Using (\ref{1in}) and (\ref{2in}) we get 
\begin{eqnarray}
\left\vert \left( u_{\alpha }^{k+1}\right) ^{\prime }(r)-\left( u_{\alpha
}^{k}\right) ^{\prime }(r)\right\vert &\leq &\frac{r^{3}\alpha }{\sigma
^{4}(N+2)k!}\left( \frac{r^{4}}{4\sigma ^{4}(N+2)}\right) ^{k}  \notag \\
&\leq &\frac{R^{3}\alpha }{\sigma ^{4}(N+2)k!}\left( \frac{R^{4}}{4\sigma
^{4}(N+2)}\right) ^{k}\overset{k\rightarrow \infty }{\rightarrow }0.
\label{derivative}
\end{eqnarray}%
Consequently,%
\begin{equation*}
(u_{\alpha }^{k}\left( r\right) )^{\prime }\overset{k\rightarrow \infty }{%
\rightarrow }(u_{\alpha }\left( r\right) )^{\prime }\text{ uniformly in }%
\left[ 0,R\right] \text{,}
\end{equation*}%
which implies that $(u_{\alpha }\left( r\right) )^{\prime }$ is a continuous
function on $\left[ 0,R\right] $. A direct computation shows that 
\begin{equation*}
u_{\alpha }\in C^{2}\left[ 0,R\right] .
\end{equation*}%
Next, let us prove (\ref{prop2}). To do this we use (\ref{1in}) succesively 
\begin{eqnarray}
u_{\alpha }^{k+1}\left( r\right) &\leq &\frac{\alpha }{\left( k+1\right) !}%
\left( \frac{r^{4}}{4\sigma ^{4}(N+2)}\right) ^{k+1}+u_{\alpha }^{k}\left(
r\right)  \notag \\
&\leq &\frac{\alpha }{\left( k+1\right) !}\left( \frac{r^{4}}{4\sigma
^{4}(N+2)}\right) ^{k+1}+\frac{\alpha }{k!}\left( \frac{r^{4}}{4\sigma
^{4}(N+2)}\right) ^{k}+u_{\alpha }^{k-1}\left( r\right)  \notag \\
&&...  \notag \\
&\leq &\underset{j=0}{\overset{k+1}{\Sigma }}\frac{\alpha }{j!}\left( \frac{%
r^{4}}{4\sigma ^{4}(N+2)}\right) ^{j}.  \label{insol}
\end{eqnarray}%
On the other hand, we note that 
\begin{equation}
u_{\alpha }\left( r\right) =\lim_{k\rightarrow \infty }u_{\alpha
}^{k+1}\left( r\right) \leq \underset{j=0}{\overset{\infty }{\Sigma }}\frac{%
\alpha }{j!}\left( \frac{r^{4}}{4\sigma ^{4}(N+2)}\right) ^{j}=\alpha e^{%
\frac{r^{4}}{4\sigma ^{4}(N+2)}},\text{ }  \label{limexp}
\end{equation}%
for all $r\in \lbrack 0,R]$.

Next, let us prove (\ref{prop3}). We observe that \{$\left( u_{\alpha
}^{k}\right) ^{\prime }(r)$\}$_{k\geq 0}$ is a nondecreasing sequence of
continuous functions. Following the proof in (\ref{insol}), and using (\ref%
{derivative}) successively it can be shown the inequality 
\begin{eqnarray*}
\left( u_{\alpha }^{k+1}\right) ^{\prime }(r) &\leq &\frac{\alpha r^{3}}{%
\sigma ^{4}(N+2)k!}\left( \frac{r^{4}}{4\sigma ^{4}(N+2)}\right) ^{k}+\left(
u_{\alpha }^{k}\right) ^{\prime }(r) \\
&&... \\
&\leq &\frac{\alpha r^{3}}{\sigma ^{4}(N+2)}\underset{j=0}{\overset{k}{%
\Sigma }}\frac{1}{j!}\left( \frac{r^{4}}{4\sigma ^{4}(N+2)}\right) ^{j}.
\end{eqnarray*}%
Repeating the arguments of (\ref{limexp}) we notice that%
\begin{eqnarray*}
\left( u_{\alpha }\right) ^{\prime }(r) &=&\lim_{k\rightarrow \infty }\left(
u_{\alpha }^{k+1}\right) ^{\prime }(r) \\
&\leq &\frac{\alpha r^{3}}{\sigma ^{4}(N+2)}\underset{j=0}{\overset{\infty}{%
\Sigma }}\frac{1}{j!}\left( \frac{r^{4}}{4\sigma ^{4}(N+2)}\right) ^{j} \\
&=&\frac{\alpha r^{3}}{\sigma ^{4}(N+2)}e^{\frac{r^{4}}{4\sigma ^{4}(N+2)}},
\end{eqnarray*}%
for all $r\in \lbrack 0,R]$.

Next, let us prove (\ref{prop4}). To do this, we observe that%
\begin{eqnarray*}
u_{\alpha }^{1}\left( r\right) &=&\alpha
+\int_{0}^{r}t^{1-N}\int_{0}^{t}s^{N+1}\frac{1}{\sigma ^{4}}u_{\alpha
}^{0}\left( r\right) dsdt \\
&=&\alpha +\int_{0}^{r}t^{1-N}\int_{0}^{t}s^{N+1}\frac{1}{\sigma ^{4}}\alpha
dsdt \\
&=&\alpha \left( 1+\frac{1}{\sigma ^{4}}\int_{0}^{r}\frac{t^{3}}{N+2}%
dt\right) \\
&=&\alpha \left( 1+\frac{1}{4\sigma ^{4}}\frac{r^{4}}{N+2}\right) .
\end{eqnarray*}%
Substituting $u_{\alpha }^{1}\left( r\right) $ into%
\begin{equation*}
u_{\alpha }^{2}\left( r\right) =\alpha
+\int_{0}^{r}t^{1-N}\int_{0}^{t}s^{N+1}\frac{1}{\sigma ^{4}}u_{\alpha
}^{1}\left( r\right) dsdt,
\end{equation*}%
we obtain%
\begin{equation*}
u_{\alpha }^{2}\left( r\right) =\alpha \left( 1+\frac{1}{4\sigma ^{4}}\frac{%
r^{4}}{\left( N+2\right) }+\frac{r^{8}}{\sigma ^{8}\cdot 4\cdot 8\cdot
\left( N+2\right) \left( N+6\right) }\right) .
\end{equation*}%
Continuing this process we get%
\begin{eqnarray*}
u_{\alpha }^{k}\left( r\right) &=&\alpha +\frac{1}{\sigma ^{4}}%
\int_{0}^{r}t^{1-N}\int_{0}^{t}s^{N+1}u_{\alpha }^{k-1}\left( s\right) dsdt
\\
&=&\alpha \left( 1+\underset{j=1}{\overset{k}{\Sigma }}\frac{1}{j!\left(
N+2\right) \left( N+6\right) ...(N+4j-2)}\left( \frac{r^{2}}{2\sigma ^{2}}%
\right) ^{2j}\right) .
\end{eqnarray*}%
Since the sequence of functions \{$u_{\alpha }^{k}$\}$_{k\geq 0}$ is uniform
convergent to the limit function $u_{\alpha }\left( r\right) $ then (\ref%
{prop4}) is proved.

The power series representation of function $u_{\alpha }\left( r\right) $
can be differentiated to obtain a power series representation of its
derivative $u_{\alpha }^{\prime }\left( r\right) $. Thus, we obtain that $%
u_{\alpha }\left( r\right) $ is differentiable on $[0,R]$ and (\ref{derivat}%
) holds true. In addition, the term-by-term derivative of a power series has
the same interval of convergence as the original power series.

Next, (\ref{derivat}) leads to $u_{\alpha }^{\prime }(0)=0$, whence (\ref%
{prop1}) is proved. A direct computation shows that 
\begin{equation*}
u_{\alpha }\in C^{2}\left( \left[ 0,R\right] \right) .
\end{equation*}%
The convexity of the solution is proved in \cite{CDPAMC} and the uniqueness
of solution follows from Remark \ref{unic}. The monotonicity of the solution
is now obvious. This completes the proof.

\textbf{Acknowledgements. }This work was supported by a mobility grant of
the Romanian Ministery of Research and Innovation, CNCS-UEFISCDI, project
number PN-III-P1-1.1-MCD-2019-0151, within PNCDI III, and NSERC grant
5-36700.

\end{document}